\documentclass[11pt,amssymb]{amsart}
\usepackage{amssymb}
\usepackage[mathscr]{eucal}

\newcommand{\Z}{{\mathbb Z}}
\newcommand{\C}{{\mathbb C}}
\newcommand{\Q}{{\mathbb Q}}

\newcommand{\G}{{\mathbb G}}
\newcommand{\g}{{\mathfrak{g}}}

\newcommand{\m}{{\mathfrak{m}}}

\newcommand{\St}{{\rm St}}
\newcommand{\Hom}{{\rm Hom}}

\newtheorem{thm}{Theorem}[section]
\newtheorem{lemma}[thm]{Lemma}
\newtheorem{prop}[thm]{Proposition}
\newtheorem{cor}[thm]{Corollary}

\usepackage[all,cmtip]{xy}
\begin{document}

\title[Abstract homomorphisms over coordinate rings of affine curves]{On abstract homomorphisms of Chevalley groups over the coordinate rings of affine curves}

\author[I.~Rapinchuk]{Igor A. Rapinchuk}

\begin{abstract}
The goal of this paper is to establish a general rigidity statement for abstract representations of elementary subgroups of Chevalley groups of rank $\geq 2$
over a class of commutative rings that includes the localizations of 1-generated rings and the coordinate rings of affine curves. This is achieved by developing the approach introduced in our paper \cite{IR}, and in particular by verifying condition (Z) of \cite{IR} over the class of rings at hand. Our main result implies, for example, that any finite-dimensional representation of $SL_n(\Z[X])$ $(n \geq 3)$ over an algebraically closed field of characteristic zero has a standard description, yielding thereby the first unconditional rigidity statement for finitely generated linear groups other than arithmetic groups/lattices.
\end{abstract}

\address{Department of Mathematics, Harvard University, Cambridge, MA, 02138 USA}

\email{rapinch@math.harvard.edu}

\maketitle

\hfill {\it To Gregory A.~Margulis on his 70th birthday}

\section{Introduction} \label{S-Intro}

The purpose of this paper is to describe finite-dimensional representations of elementary subgroups of Chevalley groups of rank $\geq 2$ over a class of commutative rings that includes the localizations of 1-generated rings and the coordinate rings of affine curves. In particular, we establish an unconditional rigidity statement for arbitrary representations
$
\rho \colon SL_n (\Z[X]) \to GL_m (K),
$
where $n \geq 3$ and $K$ is an algebraically closed field of characteristic 0 (see Corollary \ref{C-1}).

The first rigidity result for  representations of $SL_n$ $(n \geq 3)$ over
the rings of algebraic $S$-integers was obtained by Bass, Milnor, and Serre \cite{BMS} as a consequence of their solution of the congruence subgroup problem. Later, this result was subsumed by Margulis's \cite{Mar} work on the rigidity of higher rank irreducible lattices in Lie groups. More recently, there has been considerable interest in the representations and related properties (such as Kazhdan's property (T)) of $SL_n$ $(n \geq 3)$ and other Chevalley groups of rank $\geqslant 2$ over rings more general than number rings (see, e.g., \cite{EJK}, \cite{KN}, \cite{Shalom}, \cite{Sh}). In particular, motivated by a question of Kazhdan, Shenfield \cite{Sh} has shown that if $n \geq 3$, then for the so-called {\it universal lattice} $\Gamma_{n,k} := SL_n(\Z[X_1, \ldots , X_k])$ and any {\it completely reducible} representation $\rho \colon \Gamma_{n,k} \to GL_m(\C)$, one has, after passing to a suitable finite-index subgroup $\Delta \subset \Gamma_{n,k}$, a factorization
$$
\rho \vert_{\Delta} = (\sigma \circ F) \vert_{\Delta},
$$
where
$
F \colon \Gamma_{n,k} \to SL_n(\C) \times \cdots \times SL_n (\C)
$
is a group homomorphism arising from a specialization map
$
f \colon \Z[x_1, \dots, x_k] \to \C^k,
$
and $\sigma \colon SL_n (\C) \times \cdots \times SL_n (\C) \to GL_m (\C)$ is a morphism of algebraic groups.
His argument used a modification of the approach developed in \cite{BMS}, in conjunction with the fact that the congruence subgroup kernel for $\Gamma_{n,k}$ with $n \geq 3$ is central, which was proved by Kassabov and Nikolov \cite{KN} (see also \cite{CongKer} for a more general result on the centrality of the congruence kernel).

In \cite{IR}, we developed a new framework for the analysis of abstract representations of elementary subgroups of Chevalley groups over general commutative rings based on the notion of algebraic rings. This enabled us to obtain rigidity results for representations of such groups satisfying some natural conditions, which hold true trivially for completely reducible representations (giving a new proof of the result of \cite{Sh}) as well as in many other situations. In this paper, we will verify
these conditions for {\it all} representations of elementary groups over a class of commutative rings $R$ that includes the localizations of 1-generated rings (i.e. those for which there exists a surjection $\Z[X] \to R$ --- see \S\ref{S-6} for further details) as well as the coordinate rings of affine curves over number fields. To give precise statements, we need to describe our set-up.

Let $\Phi$ be a reduced irreducible root system and $G = G_{\Phi}$ the simply-connected Chevalley-Demazure group scheme of type $\Phi$ over $\Z$. For a commutative ring $R$,
denote by $E(\Phi, R)$ the {\it elementary subgroup} of $G(R)$, i.e. the subgroup generated by the $R$-points of the canonical one-parameter root subgroups $e_{\alpha} \colon \mathbb{G}_a \to G$ for all $\alpha \in \Phi.$ In \cite{IR}, we analyzed finite-dimensional representations
$$
\rho \colon E(\Phi, R) \to GL_m (K)
$$
when $K$ is an algebraically closed field of characteristic 0 and $(\Phi, R)$ is a {\it nice pair} (i.e.
$2 \in R^{\times}$ whenever $\Phi$ contains a subsystem of type $\textsf{B}_2$, and $2,3 \in R^{\times}$ if $\Phi$ is of type $\textsf{G}_2$). We showed that in many situations, such a representation
admits a {\it standard description}: namely, there exists a finite-dimensional commutative $K$-algebra $B$, together with a ring homomorphism $f \colon R \to B$ with Zariski-dense image and a morphism of algebraic groups $\sigma \colon G(B) \to GL_m(K)$, such that on a suitable finite-index $\Delta \subset E(\Phi, R)$, we have
\begin{equation}\label{E-StDesc}
\rho \vert_{\Delta} = (\sigma \circ F) \vert_{\Delta},
\end{equation}
where $F \colon E(\Phi, R) \to E(\Phi, B)$ is the group homomorphism induced by $f.$ More precisely, let $$H = \overline{\rho(E(\Phi, R))}$$ be the Zariski-closure of the image, and denote by $H^{\circ}$ and $U$ the connected component of the identity and the unipotent radical of $H^{\circ}$, respectively. As in \cite{IR}, we will say that $H^{\circ}$ {\it satisfies condition} (Z) if

\vskip2mm

\noindent {(Z) \hskip6.2cm $Z(H^{\circ}) \cap U = \{ e \}$,}

\vskip2mm

\noindent where $Z(H^{\circ})$ is the center of $H^{\circ}.$ For example, this is the case if $U$ is commutative
%, then $H^{\circ}$ satisfies condition (Z)
(see [{\it loc. cit.}, Proposition 5.5]). One of our key results in \cite{IR} is that $\rho$ does have a standard description if $H^{\circ}$ satisfies condition (Z) [{\it loc. cit.}, Theorem 6.7].

To formulate our first main result, we will need the following notation. Let $R$ be a commutative ring, $K$ a field, and $g \colon R \to K$ a ring homomorphism. We will denote by $\mathrm{Der}^g (R,K)$ the $K$-vector space of $K$-valued derivations of $R$ with respect to $g$, i.e. an element $\delta \in \mathrm{Der}^g (R,K)$ is a map $\delta \colon R \to K$ such that for any $r_1, r_2, \in R$,
$$
\delta (r_1 + r_2) = \delta (r_1) + \delta(r_2) \ \ \ {\rm and} \ \ \  \delta (r_1 r_2) = \delta(r_1) g(r_2) + g(r_1) \delta (r_2).
$$

\begin{thm}\label{T-MainThm}
Let $K$ be an algebraically closed field of characteristic 0 and $\Phi$ a reduced irreducible root system of rank $\geq 2$.
%and $G$ the simply-connected Chevalley-Demazure group scheme over $\Z$ of type $\Phi.$
Suppose $R$ is a commutative ring such that $(\Phi, R)$ is a nice pair, and assume that $\dim_K \mathrm{Der}^g (R,K) \leq 1$ for all ring homomorphisms $g \colon R \to K.$ Then for any representation
$$
\rho \colon E(\Phi, R) \to GL_m (K),
$$
$H^{\circ}$ satisfies condition {\rm (Z)}, and therefore $\rho$ has a standard description.
%Moreover, the $K$-algebra $B$ appearing in the standard description of $\rho$ is of the form
%$$
%B = B_1 \oplus \cdots \oplus B_k,
%$$
%where $B_i = K[\varepsilon_i]$, with $\varepsilon^{d_i} = 0$ for some $d_i \geq 1$.
\end{thm}

\vskip2mm

\noindent {\bf Remark 1.2.}

\vskip1mm

\noindent (a) Note that one only needs to consider representations with infinite image, as otherwise our claim is obvious. If $\rho$ has infinite image, then the $K$-algebra $B$ appearing in the standard description is of the form
$$
B = B_1 \oplus \cdots \oplus B_k,
$$
where $B_i = K[\varepsilon_i]$, with $\varepsilon_i^{d_i} = 0$ for some $d_i \geq 1$ (see Lemma \ref{L-AlgRing1} and Corollary \ref{C-1genAlgRing}). We also note that if there exists an integer $c \geq 1$ such that $cR = \{ 0 \}$, then $\rho (E(\Phi, R))$ is automatically finite for any representation $\rho$ (see \cite[Corollary 4.9]{IR}).

\vskip1mm

\noindent (b) The method of proof of Theorem \ref{T-MainThm} and Remark 3.2(b) suggest that, without additional assumptions on $R$ and $\rho$, the result is likely to be false if $\dim_K \mathrm{Der}^g (R,K) \geq 2.$

\vskip2mm

It is easy to see that if $\mathcal{O}$ is a ring of $S$-integers in a number field, then $\dim_K \mathrm{Der}^g (\mathcal{O}[X],K) \leq 1$ (see \S\ref{S-6} for the details). Furthermore, it follows from a result of Suslin (see \cite[Corollary 6.6]{Suslin}) that if $\Phi$ is of type $\textsf{A}_{\ell}$, with $\ell \geq 2$, then $E(\Phi, \mathcal{O}[X]) = SL_{\ell+1}(\mathcal{O}[X]).$ These observations lead to the following corollary, which provides the first examples of unconditional rigidity statements for finitely generated linear groups that are not arithmetic.

%Let $L$ be a number field and $\mathcal{O} \subset L$ be a ring of $S$-integers. Then one easily checks that $\mathcal{O}[X]$ is as needed (see \cite[Lemma 4.7]{IR1}).

%\vskip2mm

\addtocounter{thm}{1}

%Suppose now that $\Phi$ is of type $\textsf{A}_{\ell}$ $(\ell \geq 2)$ and $R = \Z[X].$ Then from Theorem \ref{T-MainThm}, we obtain the following generalization of a result of Shenfeld \cite{Sh}, who analyzed completely reducible complex representations of $SL_m (\Z[X_1, \dots, X_s])$ ($m \geq 3$) using considerations involving the congruence kernel in the spirit of Bass, Milnor, and Serre \cite{BMS} (see also \cite[Example 6.8]{IR}). Note that by \cite{Suslin}, $E(\textsf{A}_{m-1}, \Z[X_1, \dots, X_s]) = SL_m (\Z [X_1, \dots, X_s]).$

\begin{cor}\label{C-1}
Suppose $K$ is an algebraically closed field of characteristic 0, $\Phi$ a reduced irreducible root system of rank $\geq 2$, and $\mathcal{O}$ a ring of $S$-integers in a number field. Let $R$ be a localization of the polynomial ring $\mathcal{O}[X]$ in one variable such that $(\Phi, R)$ is a nice pair. Then any representation
$$
\rho \colon E(\Phi, R) \to GL_m (K)
$$
has a standard description. In particular, if $n \geq 3$, then any representation $$\rho \colon SL_n (\mathcal{O}[X]) \to GL_m (K)$$ has a standard description.
\end{cor}

One noteworthy feature of the examples mentioned in Corollary \ref{C-1} is that the restriction $\rho \vert_{E(\Phi, \mathcal{O})}$ is completely reducible (see, e.g., \cite[Corollary 5.2]{IR1}). It turns out that this condition implies the existence of standard descriptions in much greater generality, leading to the following ``relative" version of Theorem \ref{T-MainThm}. To state our result, we will need one additional piece of notation. Suppose $k$ is a commutative ring, $R$ a commutative $k$-algebra, and $g \colon R \to K$ a ring homomorphism. We denote by ${\rm Der}^g_k (R,K) \subset {\rm Der}^g (R,K)$ the $K$-subspace of derivations that vanish on $k.$

\begin{thm}\label{T-MainThm1}
Let $K$ be an algebraically closed field of characteristic 0, $\Phi$ a reduced irreducible root system of rank $\geq 2$,
$k$ a commutative ring such that $(\Phi, k)$ is a nice pair, and $R$ a commutative $k$-algebra. Suppose $\dim_K {\rm Der}_k^g(R, K) \leq 1$ for all ring homomorphisms $g \colon R \to K.$ Then for any representation
%Set $R = k[X]$ and suppose
$$
\rho \colon E(\Phi, R) \to GL_m (K)
$$
%is a representation
such that the restriction $\rho \vert_{E(\Phi, k)}$ is completely reducible, $H^{\circ}$ satisfies condition {\rm (Z)}, and therefore $\rho$ has a standard description.
\end{thm}

As a concrete example, let us mention the following consequence of Theorem \ref{T-MainThm1}, which is referred to in the title of the paper (see Theorem \ref{T-CurveRings} and subsequent remarks for a more general statement).

\begin{thm}\label{P-MainProp1}
Let $K$ be an algebraically closed field of characteristic 0 and $\Phi$ a reduced irreducible root system of rank $\geq 2.$
Suppose $C$ is a smooth affine irreducible curve defined over a number field $k$ with coordinate ring $R = k[C].$
%, and set $R = k[C]$ to be the algebra of $k$-defined regular functions on $C$. If $K$ is an algebraically closed field of characteristic 0 and $\Phi$ is a reduced irreducible root system of rank $\geq 2$,
Then any representation $\rho \colon E(\Phi, R) \to GL_m (K)$ has a standard description.
%If $T/R$ is an unramified ring extension, then
%any representation $$\rho \colon E(\Phi, T) \to GL_m (K)$$ for which $\overline{\rho (E(\Phi, k))}^{\circ}$ is reductive has a standard description.
\end{thm}

%\noindent Notice that this result also implies Corollary \ref{C-1} since any finite-dimensional representation of $E(\Phi, \mathcal{O})$ is completely reducible (see \cite[Corollary 5.2]{IR1}).

The proofs of Theorems \ref{T-MainThm} and \ref{T-MainThm1} proceed along similar lines and both
rely on a result asserting that certain central extensions of the group
$G(K[\varepsilon])$, where $\varepsilon^d = 0$ for some $d \geq 1$, split --- see Proposition \ref{P-SplitCE}. We will establish this statement in \S\ref{S-3} after discussing some basic facts about central extensions in \S\ref{S-2}. In \S\ref{S-4}, we quickly review the required results from our previous work on rigidity via algebraic rings, and then complete the proofs of Theorems \ref{T-MainThm} and \ref{T-MainThm1} in \S\ref{S-5}. In \S\ref{S-6}, we discuss a number of examples of rings satisfying the assumptions of our main results, which go far beyond the rings mentioned in Corollary \ref{C-1}, and include, for instance, coordinate rings of smooth affine algebraic curves defined over number fields, leading, in particular, to Theorem \ref{P-MainProp1}.
%provide some examples of rings satisfying the assumptions of our main results.
%Theorem \ref{T-MainThm}.

\vskip2mm

\noindent {\bf Notations and conventions.} We will denote by $\Phi$ a reduced irreducible root system of rank $\geq 2.$ All rings considered in this paper will be assumed to be commutative and unital. As mentioned earlier, we will say that $(\Phi, R)$ is a {\it nice pair} if $2 \in R^{\times}$ whenever $\Phi$ contains a subsystem of type $\textsf{B}_2$, and $2,3 \in R^{\times}$ if $\Phi$ is of type $\textsf{G}_2.$ Throughout the paper, we will use the standard notations $\mathbb{G}_a$ and $\mathbb{G}_m$ for the additive and multiplicative groups (over the relevant base), respectively. Finally, given a representation $\rho \colon E(\Phi, R) \to GL_n (K)$, we let $H = \overline{\rho (E(\Phi, R))}$ be the Zariski closure of the image, and set $H^{\circ}$ to be the connected component of the identity.

\vskip2mm

\noindent {\bf Acknowledgements.} I would like to thank Brian Conrad, Ofer Gabber, and Gopal Prasad for useful discussions and correspondence about central extensions. I would also like to thank Alex Lubotzky for his interest in this work. 

I began my study of rigidity questions for Chevalley groups over arbitrary rings as a graduate student at Yale University under the direction of Professor Gregory A.~Margulis. I would like to thank him for introducing me to this field, and for his support and continued interest in my work. 

I was supported by an NSF Postdoctoral Fellowship at Harvard University during the preparation of this paper.

\section{A brief review of central extensions of Chevalley groups}\label{S-2}

Let $\mathcal{G}$ be an arbitrary group. Recall that a central extension
$$
1 \to C \to E \stackrel{\pi}{\longrightarrow} \mathcal{G} \to 1
$$
is said to be {\it universal} if for any central extension
$$
1 \to C' \to E' \stackrel{\pi'}{\longrightarrow} \mathcal{G} \to 1
$$
there exists a {\it unique} homomorphism $\varphi \colon E \to E'$ making the following diagram commute
$$
\xymatrix{1 \ar[r] & C \ar[r] \ar[d]^{\varphi} & E \ar[r]^{\pi} \ar[d]^{\varphi} & \mathcal{G} \ar[r] \ar[d]^{{\rm id}} & 1 \\ 1 \ar[r] & C' \ar[r] & E' \ar[r]^{\pi'} & \mathcal{G} \ar[r] & 1}
$$
It is well-known that a group $\mathcal{G}$ admits a universal central extension if and only if it is {\it perfect} (i.e. $\mathcal{G}$ is equal to its commutator subgroup $[\mathcal{G}, \mathcal{G}]$), and the universal central extension is unique up to unique isomorphism (see, e.g., \cite[Chapter 5]{Mil}). For later use, we also note the following elementary statement.
\begin{lemma}\label{L-CE}
Let $\mathcal{G}_1$ and $\mathcal{G}_2$ be perfect groups. Suppose all central extensions of $\mathcal{G}_1$ and $\mathcal{G}_2$ split. Then any central extension of $\mathcal{G} := \mathcal{G}_1 \times \mathcal{G}_2$ also splits.
\end{lemma}
\begin{proof}
%We include the proof for completeness.
Consider a central extension
\begin{equation}\label{E-LCE1}
1 \to C \to E \stackrel{\pi}{\longrightarrow} \mathcal{G} \to 1.
\end{equation}
Restricting to each factor of $\mathcal{G}$, we obtain central extensions of $\mathcal{G}_1$ and $\mathcal{G}_2$ with splittings $\varphi_1 \colon \mathcal{G}_1 \to E$ and $\varphi_2 \colon \mathcal{G}_2 \to E.$ Then
$$
\varphi \colon \mathcal{G} \to E, \ \ \ (g_1, g_2) \mapsto \varphi_1(g_1) \varphi_2 (g_2)
$$
gives a splitting of (\ref{E-LCE1}). Indeed, it is enough to show
%For this, we only need to verify that $\varphi$ is a group homomorphism, which will follow if we show
that
$\varphi_1 (\mathcal{G}_1)$ and $\varphi_2 (\mathcal{G}_2)$ commute inside $E$.  For this we observe that for any $g_1 \in
\mathcal{G}_1$ and $g_2 \in \mathcal{G}_2$, we have $[\varphi_1(g_1) , \varphi_2(g_2)] \in C$, and then the centrality of $C$, combined with the well-known commutator identity
$
[u,v][u,w] = [u, vw][v, [w,u]],
$
implies that for every {\it fixed} $g_1 \in \mathcal{G}_1$, the map
$$
\psi_{g_1} \colon \mathcal{G}_2 \to C, \ \ g_2 \mapsto [\varphi_1(g_1) , \varphi_2(g_2)],
$$
is a group homomorphism.
%First note that since $\pi$ is a group homomorphism, we have $[\varphi_1 (g_1), \varphi_2 (g_2) ] \in C$ for all $g_1 \in \mathcal{G}_1$ and $g_2 \in \mathcal{G}_2.$ Moreover, the commutator relation
%$$
%[u,v][u,w] = [u, vw][v, [w,u]]
%$$
%and the centrality of $C$ imply that for any fixed $g_1 \in \mathcal{G}_1$, the map
%$$
%\psi \colon \mathcal{G}_2 \to C, \ \ \ g_2 \mapsto [\varphi_1 (g_1), \varphi_2 (g_2)]
%$$
%is a group homomorphism.
Since $\mathcal{G}_2$ is perfect, we obtain that $\psi_{g_1}$ is trivial for all $g_1 \in \mathcal{G}_1$, and the required fact follows.

\end{proof}

Next, let $\Phi$ be a reduced irreducible root system of rank $\geq 2$ and $G = G_{\Phi}$ be the corresponding simply-connected Chevalley-Demazure group scheme over $\Z$; for each $\alpha \in \Phi$, denote by $e_{\alpha} \colon \mathbb{G}_a \to G$ the canonical 1-parameter root subgroup. If $S$ is any commutative ring, then the elements $e_{\alpha} (s)$, for $\alpha \in \Phi$ and $s \in S$, are known to satisfy the following relations (see \cite[Chapter 3]{Stb1}):
\begin{equation}\label{E:StG103}
e_{\alpha} (s) e_{\alpha} (t) = e_{\alpha} (s+t)
\end{equation}
for all $s,t \in S$ and all $\alpha \in \Phi,$ and
\begin{equation}\label{E:StG104}
[e_{\alpha} (s), e_{\beta} (t)] = \prod e_{i \alpha + j \beta} (N^{i,j}_{\alpha, \beta} s^i t^j),
\end{equation}
for all $s,t \in S$ and all $\alpha, \beta \in \Phi,$ $\beta \neq - \alpha,$
where the product is taken over all roots of the form $i \alpha + j \beta,$ $i, j \in \Z^+$, listed in an arbitrary (but {\it fixed}) order, and the $N^{i,j}_{\alpha, \beta}$ are integers depending only on  $\Phi$ and the order of the factors in (\ref{E:StG104}), but not on the ring $S$.

Recall that the {\it Steinberg group} ${\rm St}(\Phi,S)$ is the group with generators $\tilde{x}_{\alpha} (t)$, for all $t \in S$ and $\alpha \in \Phi$, subject to relations analogous to (\ref{E:StG103}) and (\ref{E:StG104}).
%\vskip1mm
%
%(R1) $\tilde{x}_{\alpha}(s) \tilde{x}_{\alpha}(t) = \tilde{x}_{\alpha} (s+t)$
%
%\vskip1mm
%
%(R2) $[\tilde{x}_{\alpha} (s), \tilde{x}_{\beta} (t)] = \prod \tilde{x}_{i \alpha + j \beta} (N^{i,j}_{\alpha, \beta} s^i %t^j)$,
%
%\vskip1mm
%
%\noindent where $N^{i,j}_{\alpha, \beta}$ are the same integers as in (\ref{E:StG104}).
By construction, there exists a surjective group homomorphism
$
\pi_S \colon \St(\Phi, S) \to E(\Phi, S)
$
that maps $\tilde{x}_{\alpha} (t)$ to $e_{\alpha} (t)$,
and one defines
$
K_2 (\Phi, S) = \ker \pi_S.
$
We will need the following special case of a result of M.~Stein \cite{St1}.

\begin{prop}\label{P-Stein1}{\rm (cf. \cite[Corollary 4.4]{St1})}
Let $A$ be a commutative finite-dimensional algebra over a field $F$ of characteristic 0 and $\Phi$ a reduced irreducible root system of rank $\geq 2.$ Then the groups $E(\Phi, A)$ and $\St (\Phi, A)$ are perfect.
\end{prop}

\noindent It follows from Proposition \ref{P-Stein1} and the above remarks that $E(\Phi, A)$ has a universal central extension. More precisely, imitating the proof of \cite[Theorem 10, pg 78]{Stb1} given by Steinberg in the case of fields, one proves
\begin{prop}\label{P-UCEStb}
Suppose $F$ is a field of characteristic 0, $A$ a commutative finite-dimensional local $F$-algebra, and $\Phi$ a reduced irreducible root system of rank $\geq 2.$ Then
$$
1 \to K_2 (\Phi, A) \to \St (\Phi, A) \stackrel{\pi_A}{\longrightarrow} E(\Phi, A) \to 1
$$
is a universal central extension of $E(\Phi, A).$
\end{prop}

The structure of $K_2 (\Phi, S)$ has been described in terms of generators and relations by Matsumoto \cite{M1} when $S$ is an infinite field, and was later extended by van der Kallen \cite{vdK} to ``rings with many units." Before formulating the precise result that we will need, we recall some standard notations.

Let $S$ be an arbitrary commutative ring. For
any root $\alpha \in \Phi$ and $u \in S^{\times}$, we let $\tilde{w}_{\alpha} (u)$ and $\tilde{h}_{\alpha} (u)$ be the usual elements of $\St(\Phi, S)$ defined by
\begin{equation}\label{E-Stb1}
\tilde{w}_{\alpha} (u) = \tilde{x}_{\alpha} (u) \tilde{x}_{-\alpha} (-u^{-1}) \tilde{x}_{\alpha} (u) \ \ \ \text{and} \ \ \ \tilde{h}_{\alpha}(u) = \tilde{w}_{\alpha}(u) \tilde{w}_{\alpha} (-1).
\end{equation}
We also let
\begin{equation}\label{E-Stb2}
w_{\alpha} (u) = \pi_S (\tilde{w}_{\alpha} (u)) \ \ \ \text{and} \ \ \ h_{\alpha}(u) = \pi_S (\tilde{h}_{\alpha}(u))
\end{equation}
denote the corresponding elements of $E(\Phi, S).$
The {\it Steinberg symbol} associated to $\alpha \in \Phi$ is the element
$$
(u, v)_{\alpha} = \tilde{h}_{\alpha} (u) \tilde{h}_{\alpha} (v) (\tilde{h}_{\alpha} (uv))^{-1},
$$
for $u, v \in S^{\times}.$ From the presentation of Chevalley groups by generators and relations given in \cite[\S 6]{Stb1}, it is clear that all Steinberg symbols are contained in $K_2 (\Phi, S)$;
%It is easy to see that all Steinberg symbols are contained in $K_2 (\Phi, S)$ (see \cite[\S 6]{Stb1});
moreover, if $S$ is a semilocal ring with infinite residue fields, then $K_2 (\Phi, S)$ is a central subgroup of $\St (\Phi, S)$ generated by Steinberg symbols taken with respect to any long root $\alpha \in \Phi$ (see \cite[Theorem 2.13]{St2}).

Now, for an arbitrary commutative ring $S$ and any root $\alpha \in \Phi$, the Steinberg symbols are known to satisfy the following relations (see, e.g., \cite[3.2]{vdK} for references):

\vskip2mm

\noindent (a) $(x, y)_{\alpha} (xy, z)_{\alpha} = (x, yz)_{\alpha} (y,z)_{\alpha}$

\vskip1mm

\noindent (b) $(1,1)_{\alpha} = 1$

\vskip1mm

\noindent (c) $(x,y)_{\alpha} = (x^{-1}, y^{-1})_{\alpha}$

\vskip1mm

\noindent (d) $(x,y)_{\alpha} = (x, -xy)_{\alpha}$

\vskip1mm

\noindent (e) $(x,y)_{\alpha} = (x, (1-x)y)_{\alpha}$ if $x, 1-x, y \in S^{\times}$

\vskip2mm

\noindent One can show that these relations imply
$
(1, x)_{\alpha} = (x, 1)_{\alpha} = 1
$
and
$
(x,y)_{\alpha} = (y^{-1}, x)_{\alpha}
$
for all $x, y \in S^{\times}$ (notice that the former leads to the familiar relation $(x, 1-x)_{\alpha} = 1$). Furthermore, if $\Phi$ is of type different from $\textsf{C}_{\ell}$, then for any $\alpha \in \Phi,$ we have the relation

\vskip1mm

\noindent (g) $(x, yz)_{\alpha} = (x,y)_{\alpha} (x, z)_{\alpha}$.

\vskip2mm

\noindent For ease of reference, we record some special cases of Theorems 3.4 and 3.7 of \cite{vdK}.

\begin{thm}\label{T-vdK}
Suppose $S$ is a commutative semilocal ring with infinite residue fields and $\Phi$ is a reduced irreducible root system of rank $\geq 2.$ Then $K_2(\Phi, S)$ is a central subgroup of $\St(\Phi, S)$ generated by the Steinberg symbols $(x,y)_{\alpha_0}$ for some \emph{fixed} long root $\alpha_0 \in \Phi$ subject to the relations

\vskip1mm

\noindent $\bullet$ {\rm (a), (b), (c), (d), (e)} with $\alpha = \alpha_0$ if $\Phi$ is of type $\mathsf{C}_{\ell}$; and

\vskip1mm

\noindent $\bullet$ {\rm (a), (d), (e), (g)} with $\alpha = \alpha_0$ if $\Phi$ is nonsymplectic.
\end{thm}

\section{On the splitting of some central extensions}\label{S-3}

In this section, $K$ will denote a fixed algebraically closed field of characteristic 0, and we will
consider (smooth) affine algebraic groups over $K$, which we will tacitly identify with their groups of $K$-rational points. If $A$ is a finite-dimensional $K$-algebra and $G$ is a smooth affine algebraic group over $A$, we will view $G(A)$ as an algebraic group over $K$ using Weil restriction, i.e. we will implicitly use the identification  $G(A) \simeq (R_{A/K} G)(K).$

Our goal in this section is to establish the following result.

\begin{prop}\label{P-SplitCE}
Let $K$ be an algebraically closed field of characteristic 0, $\Phi$ a reduced irreducible root system of rank $\geq 2$, and $G = G_{\Phi}$ the corresponding simply-connected Chevalley-Demazure group scheme over $\Z$. Let $A = K[\varepsilon]$, with $\varepsilon^d = 0$ for some $d \geq 1.$ Then any central extension of algebraic groups over $K$ of the form
$$
1 \to W \to E \to G(A) \to 1,
$$
where $W \simeq \mathbb{G}_a^r$ is a vector group, splits.
\end{prop}

\vskip5mm

\noindent {\bf Remark 3.2.}

\vskip1mm

\noindent (a) Proposition \ref{P-SplitCE} is a particular case of the following general result stated (without proof) in \cite[Remark 5.1.5]{CGP}. Suppose $k$ is a field of arbitrary characteristic and $A$ a nonzero
%primitively generated
commutative finite-dimensional $k$-algebra generated by a single element.
Then for any simply connected semisimple $A$-group $\tilde{G}$ (with connected fibers), any central extension of $R_{A/k} (\tilde{G})$ by a $k$-group scheme of finite type splits.

\vskip3mm

\noindent (b) We should point out that the assumption in part (a) that $A$ is generated by a single element
cannot be omitted. Indeed, consider the following example. Let $\mathcal{G}$ be an algebraic group defined over an algebraically closed field $K$ of characteristic 0 and denote by $\g$ the Lie algebra of $\mathcal{G}$. Furthermore, let $A = K[\varepsilon_1, \varepsilon_2]$ with $\varepsilon_1^2 = \varepsilon_2^2 = \varepsilon_1 \varepsilon_2 = 0.$ Notice that $\mathcal{G}(A) \simeq \mathcal{G} \ltimes (\g \oplus \g),$ with $\mathcal{G}$ acting on each of the two copies of $\mathfrak{g}$ via the adjoint representation. Now let $V = \g \oplus \g \oplus K$, and define the following operation on $V$:
$$
(a_1, b_1, c_1) \cdot (a_2, b_2, c_2) = (a_1 + a_2, b_1 + b_2, c_1 + c_2 + f(a_1, b_2) - f(a_2, b_1)),
$$
where $f(x,y)$ is the Killing form on $\g.$ One easily checks that this operation is associative and in fact $(V, \cdot)$ is a group with identity element $(0,0,0)$ and $(a_1, b_1, c_1)^{-1} = (-a_1, -b_1, -c_1).$ Using the fact that $(V, \cdot)$ is stable under the adjoint action of $\mathcal{G}$, we can form the
semidirect product $E = \mathcal{G} \ltimes V.$ Notice that the central extension
$$
0 \to K \to V \to (\g \oplus \g) \to 0
$$
does not split as $V$ is noncommutative. Consequently, the central extension
$$
1 \to \mathbb{G}_a \to E \to G(A) \to 1
$$
does not split either.

\vskip5mm

\addtocounter{thm}{1}

We will now give a proof of Proposition \ref{P-SplitCE} that is based in part on discussions and correspondence with Brian Conrad, Ofer Gabber, and Gopal Prasad. We begin with the following general statement, which we will then apply in the context of the Steinberg symbols discussed in \S\ref{S-2}.

\begin{lemma}\label{L-Gopal}
Let $K$ be an algebraically closed field and set $A = K[\varepsilon]$, with $\varepsilon^d = 0$ for some $d \geq 1.$ Suppose $c \colon A^{\times} \times A^{\times} \to K^r$ is a map that arises from a morphism $R_{A/K} (\G_m) \times R_{A/K} (\G_m) \to W$, where $W \simeq \mathbb{G}_a^r$ is a vector group. Assume that $c$ satisfies
%defined by morphisms of algebraic varieties such that

\vskip1mm

\noindent {\rm (I)} $c(x, 1) = 0$ for all $x \in A^{\times}$;

\vskip1mm

\noindent {\rm (II)} $c(x,y) = c(x, (1-x)y)$ whenever $x, 1-x, y \in A^{\times}$;

\vskip1mm

\noindent {\rm (III)} $c(x,y) = c(y^{-1}, x)$ for all $x,y \in A^{\times}.$

\vskip1mm

\noindent Then $c$ is trivial (i.e. $c \equiv 0$).
\end{lemma}
\begin{proof}
We have $R_{A/K}(\mathbb{G}_m) = \mathbb{G}_m \times U$ where $U$ is a unipotent subgroup. It follows that there exists a Zariski-dense subset $X \subset R_{A/K}(\mathbb{G}_m)$ such that for $x \in X$, the element $1 - x$ lies in $A^{\times}$ and the Zariski closure of the cyclic subgroup
$\langle 1 - x \rangle$ contains $\mathbb{G}_m$, i.e. contains the {\it scalars}. Iterating (II), we obtain that $c(x , y) = c(x , (1-x)^{\ell}y)$ for any $\ell \in \Z$. So, for $x \in X$, the value $c(x , y)$ remains unchanged when an arbitrary $y \in A^{\times}$ is multiplied by a scalar. Since $X$ is Zariski-dense in $A^{\times}$, this holds
in fact for all $x , y \in A^{\times}$. It then follows from (III) that $c(x , y)$ remains unchanged when $x$ is multiplied by a scalar. (Notice that this immediately proves our claim if $d = 1$, so we will assume in the remainder of the argument that $d \geq 2$). Turning to (II) again, we see that
for any $x , y \in A^{\times}$ and any scalar $u \in \mathbb{G}_m$ we have
\begin{equation}\label{E-GopalLemma}
c(x , y) = c(x , (1 - ux)y)
\end{equation}
whenever $1 - ux \in A^{\times}$. For any $x \in A^{\times}$, we let $S(x)$ denote the closed subgroup of $A^{\times}$ generated by all elements $1 - ux \in A^{\times}$ with $u \in \mathbb{G}_m$. Then we conclude from (\ref{E-GopalLemma}) and (I) that $c(x , y) = 0$ for all $y \in S(x)$. Let $V = \{x= x_0 + x_1 \varepsilon+ \cdots + x_{d-1}\varepsilon^{d-1} \in A^{\times} \, \vert \, x_1 \neq 0\}.$ One easily shows that $S(x) = A^{\times}$ for any $x \in V.$\footnotemark
\footnotetext{This can be seen as follows. First, we observe that if $\delta = a_1 \varepsilon + \cdots + a_{d-1} \varepsilon^{d-1} \in A$ with $a_1 \neq 0$, then the elements $1 + u \delta$, where $u$ runs through almost all elements of $K$, generate $1 + \varepsilon K[\varepsilon]$ as an algebraic group. Indeed, we have $1 + \delta K[\delta] = 1 + \varepsilon K[\varepsilon].$ On the other hand, for $u_1, \dots, u_{d-1} \in K$,
$$
(1 + u_1\delta) \cdots (1 + u_{d-1}\delta) = 1 + s_1\delta + \cdots +s_{d-1} \delta^{d-1},
$$
where $s_1, \dots, s_{d-1}$ are the elementary symmetric functions in $u_1, \dots, u_{d-1}$, and the required fact follows. Now, if $x = x_0 + x_1 \varepsilon + \cdots + x_{d-1} \varepsilon^{d-1} \in A^{\times}$ with $x_1 \neq 0$, then clearly $K^{\times} \subset S(x).$ So, to show that $S(x) = A^{\times}$, we only need to verify that $1 + \varepsilon K[\varepsilon] \subset S(x).$ Set $\delta = x_1 \varepsilon + \cdots + x_{d-1} \varepsilon.$  Then
$$
1 - ux = (1 - ux_0) - u\delta = (1 - ux_0)(1 + v\delta), \ \ \ \ \text{where} \ \ v = -\frac{u}{1 - ux_0}.
$$
Thus, the elements $1 + v\delta$, for almost all $v \in K$, are contained in $S(x)$, and the first part of the argument gives the required inclusion.}
Consequently, $c(x , y) = 0$ for any $x \in V$ and $y \in A^{\times}$, and then our claim follows since $V$ is Zariski-open, hence dense in $A^{\times}$.

%It is easy to see that for a generic element $x \in A^{\times}$, the Zariski closure of the group generated by $1-x$ in $R_{A/K}(\G_m)$ contains the scalars, so (II) implies that $c(x,y)$ is unchanged when $y$ is multiplied by a scalar. It then follows from (III) that
%$c(x,y)$ is unchanged when $x$ is multiplied by a scalar. Using (II) again, we conclude that for a scalar $u \in A^{\times}$, we have $c(x,y) = c(x, (1-ux) y)$ whenever the expression makes sense. But for generic $x \in A^{\times}$, the elements $1-ux$ generate $A^{\times},$ so $c(x,y) = c(x,1) = 0,$ as needed.

\end{proof}

\vskip2mm

\noindent {\it Proof of Proposition \ref{P-SplitCE}.}  Suppose
\begin{equation}\label{E-CentExt}
1 \to W \to E \stackrel{\pi}{\longrightarrow} G(A) \to 1,
\end{equation}
where $W \simeq \mathbb{G}_a^r$,
is a central extension. Since $G(A) = E (\Phi, A)$ by \cite[Corollary 2]{M}, it follows from Proposition \ref{P-UCEStb} that there exists a unique group homomorphism $\theta \colon \text{St}(\Phi, A) \to E$ making the following diagram commute:
$$
\xymatrix{1 \ar[r] & K_2 (\Phi, A) \ar[r] \ar[d]_{\theta} & \text{St} (\Phi, A) \ar[r]^{\pi_A} \ar[d]_{\theta} & G(A) \ar[d]_{\text{id}} \ar[r] & 1 \\ 1 \ar[r] & K^r \ar[r] & E \ar[r]^{\pi} & G(A) \ar[r] & 1}
$$
To complete the proof, it suffices to show that for any root $\alpha \in \Phi$, the map
$$
\psi_{\alpha} \colon A \to E, \ \ \ r \mapsto \theta(\tilde{x}_{\alpha} (r))
$$
arises from a morphism $R_{A/K}(\mathbb{G}_a) \to E.$ Indeed, this will imply that
$$
c_{\alpha} \colon A^{\times} \times A^{\times} \to K^r, \ \ \ (u,v) \mapsto \theta ((u,v)_{\alpha}),
$$
where $(u,v)_{\alpha}$ is the Steinberg symbol associated to $\alpha$,
arises from a morphism $$R_{A/K} (\G_m) \times R_{A/K} (\G_m) \to W,$$ and is therefore trivial by Theorem \ref{T-vdK} and Lemma \ref{L-Gopal}. Consequently, $\theta$ vanishes on $K_2(\Phi, A)$, yielding an abstract section
of (\ref{E-CentExt}). Then, arguing as
%moreover, an argument analogous to the one given
in \cite[\S 6]{IR}, one shows that this section is in fact a morphism.

%To establishing the splitting of (\ref{E-CentExt}), it suffices to show that $\theta$ vanishes on $K_2 (\Phi, A).$ For this, given a root $\alpha \in \Phi$, we define
%$$
%c_{\alpha} \colon A^{\times} \times A^{\times} \to K, \ \ \ (u,v) \mapsto \theta ((u,v)_{\alpha}),
%$$
%where $(\cdot, \cdot)_{\alpha}$ is the Steinberg symbol associated to $\alpha.$ In view of Theorem \ref{T-vdK} and Lemma \ref{L-Gopal}, our claim will follow from
%Using the homomorphism $\theta,$ we can now define Steinberg symbols with values in $\G_a = K$ by
%$$
%c_{\alpha} (u, v) = \theta ((u,v)_{\alpha})
%$$
%Then, to establish the splitting of (\ref{E-CentExt}), it suffices to show that $c_{\alpha} (u, v) = 0$ for all $u,v \in A^{\times}$ and any long root $\alpha \in \Phi.$ In view of %Lemma \ref{L-Gopal}, this will follow from

\begin{lemma}
For any $\alpha \in \Phi,$ the map
$$
\psi_{\alpha} \colon A \to E, \ \ \ r \mapsto \theta(\tilde{x}_{\alpha} (r))
$$
arises from a morphism $R_{A/K}(\mathbb{G}_a) \to E.$
\end{lemma}
\begin{proof}
Let $T$ be the standard split maximal torus of $G(K)$, i.e. the group generated by the elements $h_{\alpha}(s)$ defined in (\ref{E-Stb2}) for all $\alpha \in \Phi$ and $s \in K^{\times}$ (see \cite[Theorem 6, pg 58]{Stb1}). For any root $\alpha \in \Phi$, denote by $e_{\alpha} (A)$ the corresponding root subgroup of $G(A)$, and consider the central extension
%Since $H/\G_a \simeq G(B)$, we have a lifting $\tilde{T}$ of $T$ to $H$ such that $\pi$ induces an isomorphism between $\tilde{T}$ and $T$. Next, for any root $\alpha \in \Phi$, let $U_{\alpha}(B)$ be the corresponding root subgroup of $G(B)$, and consider the extension
\begin{equation}\label{E:Ext2}
0 \to C \to E_{\alpha} = \pi^{-1}(e_{\alpha} (A)) \stackrel{\pi}{\longrightarrow} e_{\alpha} (A) \to 1,
\end{equation}
where $C \subset W$,
obtained from (\ref{E-CentExt}) by restriction. The centrality of $W$ in $E$ implies that the natural conjugation action of $G(A)$ on itself lifts to a $G(A) = E/W$-action on $E$. In particular, since $e_{\alpha}(A)$ is stable under the action of $T$ (see \cite[pg 30]{Stb1}), we conclude that (\ref{E:Ext2}) is a $T$-equivariant central extension of $e_{\alpha} (A).$ Consequently, by \cite[Lemma 5.1.6]{CGP}, there exists a $T$-equivariant morphism $\varphi_{\alpha} \colon e_{\alpha}(A) \to E_{\alpha}$ that gives a splitting of (\ref{E:Ext2}).

%(\ref{E:Ext2}) (see, for example \cite{CGP}, Lemma 5.1.6).
%it follows (\cite{CGP}, Lemma 5.1.6) that there exists a $T$-equivariant splitting $\varphi_{\alpha} \colon U_{\alpha} \to E_{\alpha}$ of (\ref{E:Ext2}).
%Next, for any $s \in A^{\times}$, let $h_{\alpha} (s) = \theta (\tilde{h}_{\alpha} (s))$ (note that $h_{\alpha} (s) \in T$ if $s \in K^{\times}$).

Now, to prove that $\psi_{\alpha}$ is regular, it is enough to show that
\begin{equation}\label{E:Ext3}
\psi_{\alpha}(r) = \varphi_{\alpha}(e_{\alpha} (r)) \ \ \ \text{for all} \ \ r \in A.
\end{equation}
For this, pick $s \in K^{\times},$ $s \neq \pm 1$, and let $t_{\alpha} \in E$ be any lift of $h_{\alpha} (s).$ Then, using the commutator relations in $G(A)$, together with the fact that $\varphi_{\alpha}$ is $T$-equivariant, we obtain
\begin{equation}\label{E:Ext4}
\varphi_{\alpha}(e_{\alpha}(r)) = \varphi_{\alpha}([h_{\alpha}(s) , e_{\alpha}(r/(s^2 - 1))]) = [t_{\alpha} , \varphi_{\alpha}(e(r/(s^2 - 1)))].
\end{equation}
On the other hand, according to the relations in the Steinberg group (cf. \cite[3.8]{St1}), we have
$$
\tilde{x}_{\alpha}(r) = [\tilde{h}_{\alpha}(s) , \tilde{x}_{\alpha}(r/(s^2 - 1))],
$$
so
\begin{equation}\label{E:Ext5}
\psi_{\alpha}(r) = \theta(\tilde{x}_{\alpha}(r)) = [\theta(\tilde{h}_{\alpha}(s)) , \theta(\tilde{x}_{\alpha}(r/(s^2 - 1)))].
\end{equation}
Since the two pairs $t_{\alpha}$, $\varphi_{\alpha}(e_{\alpha}(r/(s^2 - 1)))$ and $\theta(\tilde{h}_{\alpha}(s))$, $\theta(\tilde{x}_{\alpha}(r/(s^2 - 1)))$ consist of lifts
of the elements $h_{\alpha}(s)$, $e(r/(s^2 - 1))$ in the central extension $E$, they have the same commutator. Thus, comparing (\ref{E:Ext4}) and (\ref{E:Ext5}), we obtain (\ref{E:Ext3}), which completes the proof.

\end{proof}

\section{Rigidity and condition (Z)}\label{S-4}

To proceed with the proofs of Theorems \ref{T-MainThm} and \ref{T-MainThm1}, we need to recall the relevant technical details of the rigidity result from \cite{IR} that involves condition (Z). Consider a representation
$$
\rho \colon E(\Phi, R) \to GL_n (K),
$$
where $\Phi$ is a reduced irreducible root system of rank $\geq 2,$ $R$ a commutative ring such that $(\Phi, R)$ is a nice pair, and $K$ an algebraically closed field of characteristic 0. Then, firstly, one associates to $\rho$ a connected commutative algebraic ring $B$, together with a ring homomorphism $f \colon R \to B$ with Zariski-dense image. Since ${\rm char}~K = 0$, the algebraic ring $B$ is identified with a finite-dimensional commutative $K$-algebra (see \cite[Proposition 2.14 and Theorem 3.1]{IR}), and the Zariski density of $f(R)$ in $B$ then reduces to the fact that $f(R)$ spans $B$ over $K$. Secondly, let $H = \overline{\rho (E(\Phi, R))}$ be the Zariski closure of the image, $H^{\circ}$ the connected component of the identity, and $U$ and $Z(H^{\circ})$ the unipotent radical and center of $H^{\circ}$, respectively. One shows that if $H^{\circ}$ satisfies the following condition
\vskip2mm

\noindent (Z) \hskip6.2cm $Z(H^{\circ}) \cap U = \{ e \}$,

\vskip2mm
\noindent then there exists a morphism of algebraic groups $\sigma \colon G(B) \to H$ such that on a suitable finite-index subgroup $\Delta \subset E(\Phi, R)$, we have
$$
\rho \vert_{\Delta} = (\sigma \circ F) \vert_{\Delta},
$$
where $F \colon E(\Phi, R) \to E(\Phi, B)$ is the group homomorphism induced by $f$ (see \cite[Theorem 6.7]{IR}). The purpose of this section is to highlight the properties of $B$ and $\sigma$ that will be needed in the proofs of Theorems \ref{T-MainThm} and \ref{T-MainThm1}.

First, we have the following general statement.

\begin{lemma}\label{L-Commutator}
$H^{\circ} = [H^{\circ}, H^{\circ}].$
%The group $H^{\circ}$ coincides with its commutator subgroup $[H^{\circ}, H^{\circ}].$
\end{lemma}
\begin{proof}
This is proved in \cite[Proposition 5.3]{IR} and follows from the existence of a surjective group homomorphism ${\rm St}(\Phi, B) \twoheadrightarrow H^{\circ}$, together with the fact that ${\rm St}(\Phi, B)$ is perfect (see Proposition \ref{P-Stein1}).
\end{proof}

%Let us keep the above notations and assume that $H^{\circ}$ satisfies condition (Z). First, we note that it follows from \cite[Proposition 5.3]{IR} that $\sigma(G(B)) = H^{circ}.$ Moreover, imitating the proof of \cite[Lemma 4.2]{IR1}, we obtain

%\begin{lemma}\label{L-Isogeny}
%Assume $H^{\circ}$ satisfies condition {\rm (Z)}. Then $\sigma \colon G(B) \to H^{\circ}$ is an isogeny.
%\end{lemma}

%The next results describe the structure of $B$ in the situations considered in Theorems \ref{T-MainThm} and \ref{T-MainThm1}. if $J = J(B)$ is the Jacobson radical of a finite-dimensional commutative $K$-algebra $B$, then by the Wedderburn-Malcev theorem as above, there exists a subalgebra $\bar{B} \subset B$ such that $\bar{B} \simeq B/J \simeq K \times \cdots \times K.$

Next, let $B$ be a finite-dimensional commutative $K$-algebra and denote by $J = J(B)$ the Jacobson radical of $B$. Then, since $K$ is perfect, it follows from the Wedderburn-Malcev theorem that
%if $J = J(B)$ is the Jacobson radical of a finite-dimensional commutative $K$-algebra $B$, then by the Wedderburn-Malcev theorem as above,
there exists a $K$-subalgebra $\bar{B} \subset B$ such that $\bar{B} \simeq B/J \simeq K \times \cdots \times K$ (see \cite[Corollary 11.6]{P}). With these notations, we have the following result that describes the structure of $B$ in the situations considered in Theorems \ref{T-MainThm} and \ref{T-MainThm1}.

\begin{lemma}\label{L-AlgRing1}
Suppose $k$ is a commutative ring, $R$ a commutative $k$-algebra, and $B$ a commutative finite-dimensional algebra over an algebraically closed field $K$. Assume that $\dim_K \mathrm{Der}^g_k (R,K) \leq 1$ for all ring homomorphisms $g \colon R \to K$, where $\mathrm{Der}^g_k (R,K) \subset \mathrm{Der}^g (R,K)$ is the $K$-subspace of derivations that vanish on $k$. If
$f \colon R \to B$ is a ring homomorphism such that $f(R)$ spans $B$ over $K$ and $f(k) \subset \overline{B}$, then
%Assume that $\dim_K \mathrm{Der}^g_k (R,K) \leq 1$ for all ring homomorphisms $g \colon R \to K$, where $\mathrm{Der}^g_k (R,K) \subset \mathrm{Der}^g (R,K)$ is the $K$-subspace of derivations that vanish on $k$. Then
$$
B = B_1 \oplus \cdots \oplus B_r,
$$
where $r = \dim_K \overline{B}$ and $B_i = K[\varepsilon_i]$, with $\varepsilon_i^{d_i} = 0$ for some $d_i \geq 1$.
\end{lemma}
\begin{proof}
We begin with a reduction. Let $e_i \in \overline{B}$ be the $i$th standard basis vector. Since $e_1, \dots, e_r$ are idempotent, and we have $e_1 + \cdots + e_r = 1$ and $e_i e_j = 0$ for $i \neq j,$ it follows that we can write
$B = \oplus_{i=1}^{r} B_i,$ where $B_i = e_i B.$ Clearly, $B_i = \overline{B}_i \oplus J_i$ with $\overline{B}_i = e_i \overline{B} \simeq K$ and $J_i = e_i J$. In particular, $B_i$ is a local $K$-algebra with maximal ideal $J_i$, and we need to show that each $B_i$ has the required form. Thus, we may assume for the rest of the argument that $B = K \oplus \m$ is a finite-dimensional local $K$-algebra with maximal ideal $\m.$

%First, since $B$ is a commutative finite-dimensional $K$-algebra, in particular an artinian ring, we can write
%$
%B = B_1 \oplus \cdots \oplus B_r,
%$
%where each $B_i$ is a finite-dimensional local $K$-algebra (see \cite[Theorem 8.7]{At}), and
%We need to show that each $B_i$ has the required form.
%it suffices to show that each $B_i$ has the required form. Thus, we may assume that $B$ is a local finite-dimensional $K$-algebra with maximal ideal $\m$.

%As above, by the Wedderburn-Malcev theorem, we have
%Note that since $K$ is perfect, we have
%$B = K \oplus \m$.
%by the Wedderburn-Malcev theorem (cf. \cite[Corollary 11.6]{P}).

Now,
let $\bar{f} \colon B \to B/ \m^2$ be the composition of $f$ with the canonical map $B \to B / \m^2.$ Then,
%Let
%Let $\m_i \subset B_i$ be the maximal ideal and set $\bar{f}_i \colon R \to B_i/ \m_i^2$ to be the composition of
%$f_i \colon R \stackrel{f}{\longrightarrow} B \stackrel{p_i}{\longrightarrow} B_i,$ where $p_i$ is the projection to the $i$th component, and set $\bar{f}_i \colon R \to B_i/ \m_i^2$ to be the composition $f_i$ and
%the canonical map $B_i \to B_i/ \m_i^2.$ Since $B_i/\m_i \simeq K$ and $K$ is perfect, we have $B_i \simeq K \oplus \m_i$ (cf. \cite[Corollary 11.6]{P}). Therefore,
after choosing a $K$-basis $\{v_1, \dots, v_s \}$ of $\m / \m^2,$ we can write
$$
\bar{f} (r) = g (r) + \delta_1 (r) v_1 + \cdots + \delta_s (r) v_s,
$$
where $g \colon R \to K$ is a ring homomorphism and $\delta_1, \dots, \delta_s \in \mathrm{Der}^{g}(R,K)$; in fact, $\delta_1, \dots, \delta_s \in \mathrm{Der}^{g}_k(R,K)$ because $f(k) \subset K$.
But by our assumption, $\mathrm{Der}^{g}_k (R,K) \leq 1$,
so since $f(R)$ spans $B/ \m^2$,
%$\bar{f}$ has Zariski-dense image,
it follows that $s \leq 1.$ Consequently, by Nakayama's Lemma and \cite[Proposition 8.4]{At}, we see that $\m$ is generated by a single nilpotent element, which clearly generates $B$ over $K$, as needed.
%the fact that $B_i$ is artinian implies that $\m_i$ is a nilpotent ideal (see \cite[Proposition 8.4]{At}), and therefore, $B_i = K[\varepsilon_i]$, with $\varepsilon^{d_i} = 0$ for some $d_i \geq 1$, as claimed.
\end{proof}

Taking $k = \Z$ in the above statement, we obtain the following corollary, which will be used in the proof of Theorem \ref{T-MainThm}.

\begin{cor}\label{C-1genAlgRing}
Suppose $R$ is a commutative ring, $B$ a commutative finite-dimensional algebra over an algebraically closed field $K$, and $f \colon R \to B$ a ring homomorphism such that $f(R)$ spans $B$ over $K$.
Assume that $\dim_K \mathrm{Der}^g (R,K) \leq 1$ for all ring homomorphisms $g \colon R \to K.$ Then
$$
B = B_1 \oplus \cdots \oplus B_r,
$$
where $B_i = K[\varepsilon_i]$, with $\varepsilon_i^{d_i} = 0$ for some $d_i \geq 1$.
\end{cor}

Let us return to the notations introduced at the beginning of the section and assume that $H^{\circ}$ satisfies condition (Z). In view of our applications, we take $B$ to be a finite-dimensional commutative $K$-algebra of the form appearing in Lemma \ref{L-AlgRing1} and Corollary \ref{C-1genAlgRing}.
We have the following.

%It follows from \cite[Proposition 5.3]{IR} that $\sigma(G(B)) = H^{\circ}.$

%Moreover, imitating the proof of \cite[Lemma 4.2]{IR1}, we obtain

\begin{lemma}\label{L-Isogeny}
The morphism $\sigma \colon G(B) \to H^{\circ}$ is an isogeny.
\end{lemma}
\begin{proof}
We sketch the argument for completeness, which follows closely the proof of \cite[Lemma 4.2]{IR1}. First, \cite[Proposition 5.3]{IR} implies that $\sigma (G(B)) = H^{\circ}.$ Next, by [{\it loc. cit.}, Theorem 3.1 and Proposition 4.2], for each root $\alpha \in \Phi$, there is an injective regular map $\psi_{\alpha} \colon B \to H$ such that
$$
\sigma (e_{\alpha} (b)) = \psi_{\alpha}(b) \ \ \ \text{for all} \ \  b \in B,
$$
where $e_{\alpha}(B)$ is the 1-parameter root subgroup of $G(B)$ corresponding to $\alpha.$ By our assumption,
$$
G(B) = G(B_1) \times \cdots \times G(B_r),
$$
where $B_i = K[\varepsilon_i]$, with $\varepsilon_i^{d_i} = 0$ for some $d_i \geq 1$. Now, for each $i$, we have a Levi decomposition
$$
G(B_i) = G(K) \ltimes G(B_i, \m_i),
$$
where $G(B_i, \m_i)$ is the congruence subgroup modulo the maximal ideal $\m_i \subset B_i$, such that for every $s = 1, \dots d_i -1$, the quotient $G(B_i, \m_i^s)/G(B_i, \m_i^{s+1})$ is isomorphic as an algebraic group over $K$ to the Lie algebra $\mathfrak{g}$ of $G$ (see [{\it loc. cit}, Proposition 6.5]). It follows that if $\sigma$ is not an isogeny, then it would kill one of the simple groups $G(K)$ or $\mathfrak{g}$, and since any such group intersects each root subgroup $e_{\alpha}(B),$ this would contradict the injectivity of $\psi_{\alpha}.$
\end{proof}

%Next, let $H = \overline{\rho(E(\Phi, R))}$, and suppose the connected component $H^{\circ}$ satisfies condition (Z). Then, by \cite[Theorem 6.7]{IR}, $\rho$ has a standard description; in particular, we have a morphism of algebraic groups $\sigma \colon G(B) \to H$.
%such that
%$$
%(\sigma \circ F) \vert_{\Delta} = \rho \vert_{\Delta}
%$$
%for a suitable finite-index subgroup $\Delta \subset E(\Phi, R),$ where $F \colon E(\Phi, R) \to G(B)$ is the group homomorphism induced by $f.$
%Using \cite[Proposition 5.3]{IR} and imitating the proof of \cite[Lemma 4.2]{IR1}, we obtain

%\begin{lemma}\label{L-Isogeny}
%Assume $H^{\circ}$ satisfies condition {\rm (Z)}. Then $\sigma(G(B)) = H^{\circ}$ and $\sigma \colon G(B) \to H^{\circ}$ is an isogeny.
%\end{lemma}

\section{Proofs of Theorems \ref{T-MainThm} and \ref{T-MainThm1}}\label{S-5}

In this section, we complete the proofs of our main results. We will begin with the argument for Theorem \ref{T-MainThm}, and will then comment on the modifications needed to obtain Theorem \ref{T-MainThm1}.

Let $K$ be an algebraically closed field of characteristic 0 and $R$ a commutative ring such that $\dim_K \mathrm{Der}^g(R,K) \leq 1$ for any ring homomorphism $g \colon R \to K$. Furthermore, let $\Phi$ be a reduced irreducible root system of rank $\geq 2$ such that $(\Phi , R)$ is a nice pair, and let
$$
\rho \colon E(\Phi , R) \to GL_m(K)
$$
be a representation.
Recall that our goal is to verify that $H^{\circ}$ satisfies condition (Z).

Assume now that $H^{\circ}$ does {\it not} satisfy condition (Z), and consider the set of all closed connected subgroups $V$ of the unipotent radical $U$ of $H^{\circ}$ that are normal in $H$ and have the property that $Z(H^{\circ}/V)
\cap (U/V) \neq \{ e \}$. Pick such a subgroup $V$ of maximal dimension. Then, after replacing $H$ by $H/V$ (which is again the Zariski closure of the image of a representation of $E(\Phi , R)$), we may assume that
\begin{equation}\label{E:EEE1}
Z(H^{\circ}) \cap U \neq \{ e \},
\end{equation}
but
\begin{equation}\label{E:EEE2}
Z(H^{\circ}/W) \cap (U/W) = \{ e \}
\end{equation}
for any nontrivial subgroup $W$ of $U$ that is normal in $H$. Set $W = Z(H^{\circ}) \cap U$; this is a commutative unipotent group over a field of characteristic 0, hence a vector group (see, e.g., \cite[Proposition 15.31]{MilneAGps}).

Consider now the representation
$$
\tilde{\rho} := \tau \circ \rho \colon E(\Phi, R) \to H/W,
$$
where $\tau \colon H \to H/W$ is the canonical map. By our assumption, $Z (H^{\circ}/ W) \cap (U/W) = \{ e \}$, i.e. $H^{\circ}/ W$ satisfies condition (Z). Therefore, our rigidity result \cite[Theorem 6.7]{IR} (cf. \S\ref{S-4}) yields the existence of
%as recalled in \S\ref{S-4},
%Since $Z (H^{\circ}/ U_{\ell}) \cap (U/U_{\ell}) = \{ e \}$ by our assumption, we can apply \cite[Theorem 6.7]{IR} to conclude that
%there exists
a finite-dimensional commutative $K$-algebra $\tilde{B}$, a ring homomorphism $\tilde{f} \colon R \to \tilde{B}$ with Zariski-dense image, and a morphism of algebraic groups $\tilde{\sigma} \colon G(\tilde{B}) \to H/W$
such that for a suitable finite-index subgroup $\tilde{\Delta} \subset E(\Phi, R),$ we have
$$
\tilde{\rho} \vert_{\tilde{\Delta}} = (\tilde{\sigma} \circ \tilde{F}) \vert_{\tilde{\Delta}},
$$
where $\tilde{F} \colon E(\Phi, R) \to G(\tilde{B})$ is the group homomorphism induced by $\tilde{f}.$
%Moreover, $\tilde{\sigma}$ is an isogeny by Lemma \ref{L-Isogeny}.

%Lemma \ref{L-Isogeny} implies that $\tilde{\sigma} \colon G(\tilde{B}) \to H^{\circ}/U_{\ell}$ is an isogeny.

By construction, we have a central extension
%Using the identification $U_{\ell} \simeq \mathbb{G}_a$, we get a central extension
\begin{equation}\label{E-CentralExtension10}
1 \to W \to H^{\circ} \to H^{\circ}/W \to 1,
\end{equation}
and we let
%Pulling back along $\tilde{\sigma},$ we obtain a central extension
\begin{equation}\label{E-PullbackCE}
1 \to W \to E \to G(\tilde{B}) \to 1
\end{equation}
be the central extension obtained by taking the pullback of (\ref{E-CentralExtension10}) along the morphism $\tilde{\sigma}.$ This gives rise to the commutative diagram
$$
\xymatrix{1 \ar[r] & W \ar[r] \ar[d]_{{\rm id}} & E \ar[r] \ar[d]^{\theta} & G(\tilde{B}) \ar[r] \ar[d]^{\tilde{\sigma}} & 1 \\ 1 \ar[r] & W \ar[r] & H^{\circ} \ar[r] & H^{\circ}/W  \ar[r] & 1}
$$
where, according to Lemma \ref{L-Isogeny}, $\tilde{\sigma},$ and hence also $\theta$, are isogenies.
%$\theta$ and $\tilde{\sigma}$ are isogenies by Lemma \ref{L-Isogeny}.

Now, according to Corollary \ref{C-1genAlgRing}, we have
%our assumption that $\dim_K \mathrm{Der}^g (R,K) \leq 1$ for all $g \colon R \to K,$ implies that
$$
\tilde{B} = \tilde{B}_1 \oplus \cdots \oplus \tilde{B}_r,
$$
where $\tilde{B}_i = K[\varepsilon_i]$, with $\varepsilon_i^{d_i} = 0$ for some $d_i \geq 1.$
Since $E(\Phi, \tilde{B}_i) = G(\tilde{B}_i)$ by \cite[Corollary 2]{M}, it follows from Proposition \ref{P-Stein1} that $G(\tilde{B}_i)$ is perfect. Consequently,  we conclude from Lemma \ref{L-CE} and Proposition \ref{P-SplitCE} that (\ref{E-PullbackCE}) splits as a sequence of abstract groups. Then the abstract commutator $[E,E]$ has trivial intersection with $W$. But $[E , E]$ is also Zariski-closed (see \cite[Corolary 2.3]{Bo}), so since $W \neq \{ e \}$, we see that $\dim  [E , E] < \dim E$. Then, since $\theta$ is an isogeny, it follows that the dimension of
$[H^{\circ} , H^{\circ}] = \theta([E , E])$ is strictly smaller than $\dim H^{\circ} = \dim E$. This contradicts Lemma \ref{L-Commutator}, proving that $H^{\circ}$ satisfies condition (Z) and completing the proof of Theorem \ref{T-MainThm}.
%In particular, $\dim [E,E] < \dim E.$ On the other hand, $H^{\circ} = [H^{\circ}, H^{\circ}]$ (by Lemma \ref{L-Commutator}) and $\tilde{\sigma}$ is an isogeny (by Lemma \ref{L-Isogeny}).
%by Lemmas \ref{L-Commutator} and \ref{L-Isogeny}, respectively.
%A contradiction, proving that $H^{\circ}$ satisfies condition (Z), and completing the proof of Theorem \ref{T-MainThm}.
$\Box$

\vskip3mm

We now turn to the proof of Theorem \ref{T-MainThm1}, which proceeds along similar lines. As above, $K$ will be an algebraically closed field of characteristic 0 and $\Phi$ a reduced irreducible root system of rank $\geq 2.$ Suppose $k$ is a commutative ring such that $(\Phi, k)$ is a nice pair and $R$ is a commutative $k$-algebra. Assume that $\dim_K {\rm Der}^g_k(R,K) \leq 1$ for all ring homomorphisms $g \colon R \to K,$ and let
%set $R = k[X]$, and let
$$
\rho \colon E(\Phi, R) \to GL_m(K)
$$
be a representation such that the restriction $\nu = \rho \vert_{E(\Phi, k)}$ is completely reducible (so that $\overline{\nu (E(\Phi, k))}^{\circ}$ is reductive). The goal is again to verify that $H^{\circ}$ satisfies condition (Z). Let us suppose that is not the case. Then, arguing as in the proof of Theorem \ref{T-MainThm}, we may assume that conditions (\ref{E:EEE1}) and (\ref{E:EEE2}) are satisfied.
We let $W = Z(H^{\circ}) \cap U$ and consider the representation
$$
\tilde{\rho} = \tau \circ \rho \colon E(\Phi, R) \to H/W,
$$
where $\tau \colon H \to H/W$ is the canonical map. If $\tilde{\nu}$ denotes the restriction $\tilde{\rho} \vert_{E(\Phi, k)}$, then again $\overline{\tilde{\nu}(E(\Phi, k))}^{\circ}$ is reductive. Let $\tilde{B}_0 \subset \tilde{B}$ be the $K$-algebras associated to $\tilde{\nu}$ and $\tilde{\rho}$, respectively. Then
$$
\tilde{B}_0 \simeq K \times \cdots \times K
$$
(see \cite[Lemma 5.7]{IR}), so, using Lemma \ref{L-AlgRing1}, we conclude
%Consequently, using Lemma \ref{L-AlgRing1}, together with \cite[Lemma 5.7 and Proposition 6.5]{IR}, we conclude
%\cite[Lemma 5.7]{IR}, the finite-dimensional $K$-algebra $\tilde{B}_{\tilde{\nu}}$ associated to $\tilde{\nu}$ has the form
%$$
%\tilde{B}_{\tilde{\nu}} = K \times \cdots \times K.
%$$
%It then follows from Lemma \ref{L-AlgRing1}
that %the $K$-algebra $\tilde{B}$ associated to $\tilde{\rho}$ is of the form
$$
\tilde{B} = \tilde{B}_1 \oplus \cdots \oplus \tilde{B}_r,
$$
where $\tilde{B}_i = K[\varepsilon_i]$, with $\varepsilon_i^{d_i} = 0$ for some $d_i \geq 1.$ The same argument as above then leads to a contradiction, proving that $H^{\circ}$ satisfies condition (Z), as claimed.

%This is a contradiction since $\tilde{\sigma}$ is an isogeny and, as noted above, $H^{\circ} = [H^{\circ}, H^{\circ}]$, which completes the proof.

\section{Examples and applications}\label{S-6}

In this section, we discuss examples of rings satisfying the hypotheses of Theorems \ref{T-MainThm} and \ref{T-MainThm1}.
We continue to denote by $K$ an algebraically closed field of characteristic 0.

As we already mentioned in \S\ref{S-Intro}, some initial examples of rings satisfying the assumptions of Theorem \ref{T-MainThm} are 1-generated rings and their localizations. More generally, we have the following.

\begin{lemma}\label{L-Examples1}
Let $\mathcal{O}$ be an integral ring extension of $\Z$,
%be a subring of the ring of integers of a number field,
$A$ be a quotient of the polynomial ring $\mathcal{O}[X]$ in one variable, and $R$ be any localization of $A$. Then $\dim_K {\rm Der}^g(R,K) \leq 1$ for any ring homomorphism $g \colon R \to K.$
%for any ring homomorphism $g \colon R \to K$, we have $\dim_K {\rm Der}^g(R,K) \leq 1.$
\end{lemma}
\begin{proof}
This is elementary, so we only give a sketch. First, any ring homomorphism $f \colon \mathcal{O} \to K$ is injective and ${\rm Der}^f (\mathcal{O}, K) = 0$. It follows that $\dim_K {\rm Der}^g (A, K) \leq 1$ for any ring homomorphism $g \colon A \to K.$ On the other hand, for any ring homomorphism $g \colon R \to K$, the restriction ${\rm Der}^g (R,K) \to {\rm Der}^{g \vert_A} (A, K)$ is injective, yielding our claim.
\end{proof}

\vskip3mm

\noindent {\it Proof of Corollary \ref{C-1}.} Let $\mathcal{O}$ be a ring of $S$-integers in a number field, and let $R$ be a localization of $\mathcal{O}[X]$ such that $(\Phi, R)$ is a nice pair. Then Theorem \ref{T-MainThm}, in conjunction with Lemma \ref{L-Examples1}, gives the first assertion. Furthermore,
%Let us now take $\mathcal{O}$ to be a ring of $S$-integers in a number field. Then, using Lemma \ref{L-Examples1}, we can easily derive Corollary \ref{C-1} from Theorem \ref{T-MainThm}. Indeed, we have just shown that $\dim_K {\rm Der}^g (\mathcal{O}[X], K) \leq 1$ for all $g \colon \mathcal{O}[X] \to K$, which gives the first statement of Corollary \ref{C-1}.
since $SK_1 (\mathcal{O}) = 0$ and $\mathcal{O}$ has Krull dimension 1, a result of Suslin (see \cite[Corollary 6.6]{Suslin}) guarantees that $SL_n (\mathcal{O}[X])$ is generated by elementary matrices for all $n \geq 3$, yielding the second statement. $\Box$

Next, we mention a relative version of Lemma \ref{L-Examples1}, whose proof is similar.

\begin{lemma}\label{L-Examples2}
Let $k$ be a commutative ring and $R$ be any localization of a quotient of the polynomial ring $k[X]$. Then $\dim_K {\rm Der}^g_k (R,K) \leq 1$ for any ring homomorphism $g \colon R \to K.$
\end{lemma}

Combining this statement with Theorem \ref{T-MainThm1}, we obtain the following.

\begin{cor}\label{C-Examples1}
Let $K$ be an algebraically closed field of characteristic 0, $\Phi$ a reduced irreducible root system of rank $\geq 2$, and $k$ a commutative ring. Suppose $R$ is a localization of a quotient of $k[X]$. Denote by $k'$ the image of $k$ in $R$, and assume that $(\Phi, k')$ is a nice pair. Then any representation $\rho \colon E(\Phi, R) \to GL_m (K)$ such that $\overline{\rho(E(\Phi, k'))}^{\circ}$ is reductive has a standard description.
\end{cor}

\vskip3mm

\noindent {\bf Remark 6.4.} Let $\mathcal{O}$ be a ring of $S$-integers in a number field. Then, as we remarked above, ${\rm Der}^g (\mathcal{O},K) = 0$ for any ring homomorphism $g \colon \mathcal{O} \to K$, which in fact implies that if $\Phi$ is a reduced irreducible root system of rank $\geq 2$ and $(\Phi, \mathcal{O})$ is a nice pair, then any representation $\rho \colon E(\Phi, \mathcal{O}) \to GL_m(K)$ is automatically completely reducible (see \cite[Corollary 5.2]{IR1}). Thus, Corollary \ref{C-Examples1} subsumes the first assertion of Corollary \ref{C-1}.

\addtocounter{thm}{1}

\vskip3mm

We will now show that the class of rings in question is much larger than those obtained from the rings of polynomials in one variable by taking quotients and localizations and includes, for example, coordinate rings of smooth affine curves.
\begin{prop}\label{P-RingCurve}
Suppose $C$ is a smooth irreducible affine curve defined over a subfield $k \subset K$ with coordinate ring $R = k[C].$
%, and let $R = k[C]$ be the algebra of $k$-defined regular functions.
Then $\dim_K {\rm Der}^g_{k} (R,K) \leq 1$ for any homomorphism $g \colon R \to K.$
\end{prop}
\begin{proof}
Set $\mathfrak{p} = \ker g$ and let $R_{\mathfrak{p}}$ be the corresponding localization of $R$. Then clearly
$$
{\rm Der}^g_k (R, K) = {\rm Der}^{\tilde{g}}_k (R_{\mathfrak{p}}, K),
$$
where $\tilde{g}$ denotes the localization of $g$. On the other hand, $R_{\mathfrak{p}}$ can be identified with the local ring $\mathcal{O}_{C,x}$ of $C$ at $x$, where $x$ is the point of $C$ corresponding to $\mathfrak{p}$, and
$$
{\rm Der}^{\tilde{g}}_k(\mathcal{O}_{C,x}, K) = \Hom_{\mathcal{O}_{C,x}-{\rm mod}}(\Omega_{\mathcal{O}_{C,x}/k}, K),
$$
where $\Omega_{\mathcal{O}_{C,x}/k}$ is the module of relative differentials of $\mathcal{O}_{C,x}$ over $k$, and $K$ is endowed with the structure of an $\mathcal{O}_{C,x}$-module via $\tilde{g}.$ Since $C$ is smooth of dimension 1, $\Omega_{\mathcal{O}_{C,x}/k}$ is a free $\mathcal{O}_{C,x}$-module of rank 1 (see, e.g., \cite[\S 8.5, Proposition 5]{Bosch}), which yields our claim.

\end{proof}

\vskip3mm

\noindent {\bf Remark 6.6.}

\noindent (a) While Proposition \ref{P-RingCurve} remains valid for {\it all} irreducible affine curves $C$ if $g$ is {\it injective}, the following example shows that it may fail for non-smooth curves if $g$ is no longer injective. Let $C/\Q$ be the affine plane curve given by the equation $X^3 - Y^2 = 0.$ Then $R = \Q[X,Y]/(X^3 - Y^2),$ and if $g \colon R \to \overline{\Q}$ is the ring homomorphism induced from $g_0 \colon \Q[X,Y] \to \overline{\Q}$ sending $X$ and $Y$ to 0, then
one easily checks that $\dim_{\overline{\Q}}{\rm Der}^g_{\Q} (R, \overline{\Q}) = 2$ (see \cite[Remark 4.8]{IR1}).

\vskip1mm

\noindent (b) If $k$ is a number field and $\mathcal{O} \subset k$ is a ring of $S$-integers, we can further generalize Proposition \ref{P-RingCurve} as follows. Suppose $C \subset \mathbb{A}_K^t$, and let $c_1, \dots, c_t$ denote the images in $k[C]$ of the coordinate functions on $\mathbb{A}_k^t.$ Set $R \subset k[C]$ to be the $\mathcal{O}$-algebra generated by $c_1, \dots, c_t.$ Then using the fact that any ring homomorphism $f \colon \mathcal{O} \to K$ is injective and ${\rm Der}^f (\mathcal{O}, K) = 0$, one shows that $\dim_K {\rm Der}^g (R, K) = \dim_K {\rm Der}^g_{\mathcal{O}} (R, K) = 1$ for any homomorphism $g \colon R \to K$. In particular, taking $C = \mathbb{A}_K^1$, we recover Lemma \ref{L-Examples1}.

\vskip1mm

\noindent (c) We should point out that one obtains further examples by considering unramified extensions of the rings in Proposition \ref{P-RingCurve}. We recall that a ring homomorphism $R \to T$ of finite type is said to be {\it unramified} if $\Omega_{T/R} = 0$ (see, e.g.,  \cite[Commutative algebra, \S 147]{Stacks} for a detailed discussion). It is a straightforward consequence of the definition that if $R$ is as in Proposition \ref{P-RingCurve} and $T/R$ is an unramified ring extension, then for any ring homomorphism $g \colon T \to K$, we have an embedding ${\rm Der}^{g}_k (T, K) \hookrightarrow {\rm Der}^{g \vert_R}_k (R,K),$ hence $\dim_K {\rm Der}^{g}_k (T, K) \leq 1$ (see, e.g., \cite[\S 8.1, Proposition 12]{Bosch}).

\vskip3mm

\addtocounter{thm}{1}

In conclusion, Theorem \ref{T-MainThm1}, in conjunction with Proposition \ref{P-RingCurve}, leads to the following rigidity result.

\begin{thm}\label{T-CurveRings}
Let $K$ be an algebraically closed field of characteristic 0 and $\Phi$ a reduced irreducible root system of rank $\geq 2.$ Suppose $C$ is a smooth irreducible affine curve defined over a field $k$ of characteristic 0 and set $R = k[C]$ to be the coordinate ring of $C$. If $\rho \colon E(\Phi, R) \to GL_m (K)$ is a representation such that $\overline{\rho (E(\Phi, k))}^{\circ}$ is reductive, then $\rho$ has a standard description.
\end{thm}

We note that if $k$ is a number field, then $\overline{\rho (E(\Phi, k))}^{\circ}$ is automatically reductive (see Remark 6.4), leading to an unconditional rigidity statement in this case (cf. Theorem \ref{P-MainProp1}). We also have a similar unconditional result if we take $R$ to be the $\mathcal{O}$-algebra discussed in Remark 6.6(b) and $(\Phi, R)$ is a nice pair.

%or if $R$ is replaced by the $\mathcal{O}$-algebra discussed in Remark 6.6(a), then $\overline{\rho (E(\Phi, k))}^{\circ}$ and $\overline{\rho (E(\Phi, \mathcal{O}))}^{\circ}$ are automatically reductive (see Remark 6.4), leading to unconditional rigidity statements in these cases.

%\vskip1mm

%\noindent $\bullet$ The composition of unramified homomorphisms is again unramified.

%\vskip1mm

%\noindent $\bullet$ Unramified homomorphisms are stable under base change, i.e. if $R \to T$ is unramified and $R \to R'$ is a ring homomorphism, then $R' = R \otimes_R R' \to T \otimes_R R'$ is unramified.

%\vskip1mm

%\noindent $\bullet$ For any nonzero $v \in R$, the localization $R_v$ is an unramified extension of $R$.

%
%\vskip1mm

%\noindent $\bullet$ If $T_1/R, \dots, T_r/R$ are unramified ring extensions, then $T_1 \otimes_R \cdots \otimes_R T_r/R$ is unramified as well. %(see, e.g., \cite[\S 8.1, Proposition 14]{Bosch}).

%\vskip2mm

\bibliographystyle{amsplain}

\end{document}